\documentclass[12pt]{amsart}

\usepackage{amsmath,amsfonts,amssymb,amsthm,array}
\usepackage{url}
\usepackage[colorlinks,linkcolor=blue,anchorcolor=blue,citecolor=blue]{hyperref}

\newtheorem{thm}{Theorem}
\newtheorem{lem}[thm]{Lemma}
\newtheorem{prop}[thm]{Proposition}

\def \Z{{\mathbb Z}}
\def \Q{{\mathbb Q}}

\def \wH{{\mathrm H}}
\def\Gal{{\rm Gal}}

\def\cA{{\mathcal A}}
\def\cB{{\mathcal B}}
\def\cC{{\mathcal C}}
\def\cE{{\mathcal E}}
\def\cF{{\mathcal F}}
\def\cM{{\mathcal M}}

\def\card#1{\left|#1\right|}

\def\({\left(}
\def\){\right)}
\def\mand{\qquad \mbox{and} \qquad}
\def\ssum{\mathop{\sum\, \ldots \sum}}

\numberwithin{equation}{section}
\numberwithin{thm}{section}
\begin{document}

\title[Multiplicatively dependent vectors]{On multiplicatively dependent vectors of algebraic numbers}

\author[F.~Pappalardi]{Francesco Pappalardi}
\address{Dipartimento di Matematica e Fisica, Universit\`a Roma Tre, Roma, I--00146, Italy}
\email{pappa@mat.uniroma3.it}

\author[M.~Sha]{Min Sha}
\address{School of Mathematics and Statistics, University of New South Wales,
 Sydney, NSW 2052, Australia}
\email{shamin2010@gmail.com}

\author[I. E.~Shparlinski]{Igor E. Shparlinski}
\address{School of Mathematics and Statistics, University of New South Wales,
 Sydney, NSW 2052, Australia}
\email{igor.shparlinski@unsw.edu.au}

\author[C. L.~Stewart]{Cameron L. Stewart}
\address{Department of Pure Mathematics, University of Waterloo,
Waterloo, Ontario, N2L 3G1, Canada}
\email{cstewart@uwaterloo.ca}



\begin{abstract}
In this paper, we give several asymptotic formulas for the number of multiplicatively dependent vectors of  algebraic numbers of fixed degree, or within a fixed number field, and bounded height. 
\end{abstract}

\maketitle

\section{Introduction} %Section 1

\subsection{Background}

Let $n$ be a positive integer, $G$ be a multiplicative group and let $\pmb{\nu}=(\nu_1,\ldots,\nu_n)$ be in $G^n$. We say that $\pmb{\nu}$ is multiplicatively dependent if there is a non-zero vector $\mathbf{k}=(k_1,\ldots,k_n)\in \mathbb{Z}^n$ for which
\begin{equation} \label{eq:MultDep}
\pmb{\nu}^{\mathbf{k}}=\nu^{k_1}_1\cdots \nu^{k_n}_n=1.
\end{equation}
We denote by $\cM_n(G)$ the set of multiplicatively dependent vectors in $G^n$.

For instance, the set $\cM_n(\mathbb{C}^*)$ of multiplicatively dependent vectors in $(\mathbb{C}^*)^n$ is of Lebesgue measure zero, since it is a countable union of sets of measure zero. Further, if we fix an exponent vector $\mathbf{k}$ the subvariety of $(\mathbb{C}^*)^n$ determined by~\eqref{eq:MultDep} is an algebraic subgroup of $(\mathbb{C}^*)^n$.

For multiplicatively dependent vectors of algebraic numbers  there are two kinds of questions which have been extensively studied. 
The first question concerns the exponents in~\eqref{eq:MultDep}. Given a multiplicatively dependent vector $\pmb{\nu}$  it follows from the work of Loxton and van der Poorten~\cite{LvdP,vdPL}, Matveev~\cite{Matveev},  and Loher and Masser~\cite[Corollary~3.2]{LM} (attributed to K.~Yu) that there is a relation of the form~\eqref{eq:MultDep} with a non-zero vector $\mathbf{k}$ with small coordinates. The second question is to find comparison relations among the heights of the coordinates. For example, 
Stewart~\cite[Theorem~1]{Stewart} has given an inequality for the heights of the coordinates of such a vector (of low multiplicative rank, in the terminology of Section~\ref{sec:rank}), and a lower bound for the sum of the heights of the coordinates is implied in~\cite{Vaaler}.  

In this paper, we obtain severa asymptotic formulas for the number of multiplicatively dependent $n$-tuples whose coordinates are algebraic numbers of fixed degree, or within a fixed number field, and bounded height.  
Aside from the results mentioned above, to the best of our knowledge, this natural question has never been addressed in the literature. 

We remark that the above question is interesting in its own right, but is also partially motivated by 
the works~\cite{OstSha,SV},   where multiplicatively independent vectors play an important role.

\subsection{Rank of multplicative independence}
\label{sec:rank}

The following notion plays a crucial role in our argument, and is also of independent interest. 

Let $\overline{\Q}$ be an algebraic closure of the rational numbers $\Q$. For each $\pmb{\nu}$ in $(\overline{\Q}^*)^n$, we define $s$, the \textit{multiplicative rank} of $\pmb{\nu}$, in the following way. If $\pmb{\nu}$ has a coordinate which is a root of unity, we put $s=0$; otherwise let $s$ be the largest integer with $1\leq s\leq n$ for which any $s$ coordinates of $\pmb{\nu}$ form a multiplicatively independent vector. Notice that
\begin{equation} \label{mn}
0\leq s\leq n-1,
\end{equation}
whenever $\pmb{\nu}$ is multiplicatively dependent.

\subsection{Conventions and notation}
\label{sec:con}

For any algebraic number $\alpha$, let
$$
f(x)=a_dx^d+\cdots+a_1x+a_0
$$
be the minimal polynomial of $\alpha$ over the integers $\Z$  (so with content $1$ and positive leading coefficient). Suppose that $f$ is factored as
$$
f(x)=a_d(x-\alpha_1)\cdots (x-\alpha_d)
$$
over the complex numbers $\mathbb{C}$. The \textit{naive height} $\wH_0(\alpha)$ of $\alpha$ is given by
$$
\wH_0(\alpha)=\max\{|a_d|,\ldots,|a_1|,|a_0|\} ,
$$
and $\wH(\alpha)$, the height of $\alpha$, also known as the \textit{absolute Weil height} of $\alpha$, is defined by
$$
\wH(\alpha)=\(a_d\prod^d_{i=1}\max\{1,|\alpha_i|\}\)^{1/d}.
$$

Let $K$ be a number field of degree $d$ (over $\Q$).
We use the following standard notation:
\begin{itemize}
\item $r_1$ and $r_2$ for the number of real and non-real embeddings of $K$, respectively, and put $r=r_1+r_2-1$;
\item $D,h, R$ and  $\zeta_K$  for the discriminant, class number,  regulator
and  Dedekind zeta function  of $K$, respectively;
\item   $w$ for the number of roots of unity in $K$.
\end{itemize}
Note that $r$ is exactly the rank of the unit group of the ring of algebraic integers of $K$. As usual, let $\zeta(s)$ be the Riemann zeta function.

For any real number $x$, let $\lceil x\rceil$ denote the smallest integer greater than or equal to $x$, and let $\lfloor x\rfloor$ denote the greatest integer less than or equal to $x$.

We always implicitly assume that $H$ is large enough, in particular so that the logarithmic 
expressions $\log H$ and $\log \log H$ are well-defined. 

In the sequel, we use the Landau symbols $O$ and $o$ and the Vinogradov symbol $\ll$. We recall that the assertions $U=O(V)$ and $U\ll V$  are both equivalent to the inequality $|U|\le cV$ with some positive constant $c$, while $U=o(V)$ means that $U/V\to 0$.
We also use the asymptotic notation $\sim$.

For  a finite set $S$ we use $\card{S}$ to denote its cardinality. 

Throughout the paper, the implied constants in the symbols $O$ and $\ll$ only depend on the given number field $K$, the given degree $d$, or the dimension $n$.

\subsection{Counting vectors within a number field} 
\label{sec:fixK}

Let $K$ be a number field of degree $d$. Denote the set of algebraic integers of $K$ of height at most $H$ by $\cB_K(H)$ and the set of algebraic numbers of $K$ of height at most $H$ by $\cB^*_K(H)$. 
Set
$$
B_K(H)=\card{\cB_K(H)} \mand B^*_K(H) = \card{\cB^*_K(H)}.
$$

Put
$$
C_1(K)=\frac{2^{r_1}{(2\pi)}^{r_2}d^r}{|D|^{1/2}r!}.
$$
It follows directly from the work of Widmer~\cite[Theorem~1.1]{Widmer2} (taking $n=e=1$ there)  that
\begin{equation} \label{BK}
B_K(H)=C_1(K)H^d(\log H)^r+O\(H^d(\log H)^{r-1}\).
\end{equation}
If $r=0$, then~\eqref{BK} can be improved to (see~\cite[Theorem~1.1]{Barroero1})
\begin{equation} \label{BK0}
B_K(H)=C_1(K)H^d+O(H^{d-1}).
\end{equation}
We remark that the estimate in~\eqref{BK} is stated in~\cite[Chapter~3, Theorem~5.2]{Lang} without the explicit constant $C_1(K)$, and moreover Barroero~\cite{Barroero2} has  obtained similar estimates for the number of algebraic $S$-integers with fixed degree and bounded height.

Define
$$ 
C_2(K)=\frac{2^{2r_1}(2\pi)^{2r_2}2^rhR}{|D|w \zeta_K(2)}.
$$
Schanuel~\cite[Corollary to Theorem~3]{Schanuel} proved in 1979 (see also~\cite[Equation~(1.5)]{Masser2}) that
\begin{equation} \label{B*K}
B^*_K(H)=C_2(K)H^{2d}+O\( H^{2d-1}(\log H)^{\sigma(d)} \),
\end{equation}
where $\sigma(1)=1$ and $\sigma(d)=0$ for $d>1$.
Note that the height in~\cite{Schanuel} is our height to the power $d$.

For any positive integer $n$, we denote by $L_{n,K}(H)$ the number of multiplicatively dependent $n$-tuples whose coordinates are algebraic integers of height at most $H$, and we denote by $L^*_{n,K}(H)$ the number of multiplicatively dependent $n$-tuples whose coordinates are algebraic numbers of height at most $H$.

Put
$$ 
C_3(n,K)=\frac{n(n+1)}{2}w C_1(K)^{n-1}.
$$

\begin{thm} 
\label{thm:MnK}
Let $K$ be a number field of degree $d$ over $\Q$ and let $n$ be an integer with $n\geq 2$. We have
\begin{equation} \label{MnK}
\begin{split}
L_{n,K}(H)=C_3(n,K)&H^{d(n-1)}(\log H)^{r(n-1)}  \\
&+O\(H^{d(n-1)}(\log H)^{r(n-1)-1}\);
\end{split}
\end{equation}
%if  furthermore $K=\Q$ or $K$ is an imaginary quadratic field, we have  
%\begin{equation} \label{MnQ}
%L_{n,K}(H)=C_3(n,K)H^{d(n-1)}+O\(H^{d(n-3/2)}\).
%\end{equation}
if  furthermore   $K=\Q$ or is an imaginary quadratic field, we have  
\begin{equation} \label{MnQ}
L_{n,K}(H)=C_3(n,K)H^{d(n-1)}+O\(H^{d(n-3/2)}\).
\end{equation}
\end{thm}

We remark that when $K=\Q$ a better error term than that given in~\eqref{MnQ} is 
stated in Theorem~\ref{thm:Mnd} below, more precisely, see~\eqref{Mnd'}.

We  estimate $L^*_{n,K}(H)$ next. Put
$$
C_4(n,K)=n^2 w C_2(K)^{n-1}.
$$

\begin{thm} 
\label{thm:MnK2}
Let  $K$ be a number field of degree $d$, and let $n$ be an integer with $n\geq 2$. 
Then, we have 
\begin{equation}
L^*_{n,K}(H) =C_4(n,K)H^{2d(n-1)}
+O\(H^{2d(n-1)-1}g(H)\),
\end{equation}
where 
\begin{equation*}
g(H)=\left\{ \begin{array}{ll}
      \log H & \textrm{if $d=1$ and $n=2$}\\
      \exp(c\log H/\log\log H) & \textrm{if $d=1$ and $n>2$}\\
      1 & \textrm{if $d>1$ and $n\ge 2$},
                 \end{array} \right.
\end{equation*}
and $c$ is a positive number depending only on $n$. 
\end{thm}

%We remark that if $n=2$, then Theorem~\ref{thm:MnK2} holds with an error term which is $O\( H^{2d-1}(\log H)^{\sigma(d)} \)$ provided that one appeals to~\eqref{B*K} for the proof.

We now outline the strategy of the proofs. Given a number field $K$, we define $L_{n,K,s}(H)$ and $L^*_{n,K,s}(H)$ to be the number of multiplicatively dependent $n$-tuples of multiplicative rank $s$ whose coordinates are algebraic integers in $\cB_K(H)$ and algebraic numbers in $\cB^*_K(H)$ respectively. It follows from~\eqref{mn} that
\begin{equation} \label{MnK=}
\left\{ \begin{array}{ll}
                 L_{n,K}(H)=L_{n,K,0}(H)+\cdots+L_{n,K,n-1}(H)\\
                 \\
                 L^*_{n,K}(H)=L^*_{n,K,0}(H)+\cdots+L^*_{n,K,n-1}(H).
                 \end{array} \right.
\end{equation}
The main term in~\eqref{MnK} comes from the contributions of $L_{n,K,0}(H)$ and $L_{n,K,1}(H)$ in~\eqref{MnK=}, and the main term in Theorem~\ref{thm:MnK2} comes from the contributions of  $L^*_{n,K,0}(H)$ and $L^*_{n,K,1}(H)$ in~\eqref{MnK=}.
To prove Theorems~\ref{thm:MnK}  and~\ref{thm:MnK2},  we make use of~\eqref{MnK=} and the following result.

\begin{prop} 
\label{prop:MnKm}
Let $K$ be a number field of degree $d$. Let $n$ and $s$ be   integers with $n\geq 2$ and  $0 \le s\leq n-1$.  
Then, there exist positive numbers $c_1$ and $c_2$ which depend on $n$ and $K$, such that 
\begin{equation} \label{MnKm}
L_{n,K,s}(H)<H^{d(n-1)-d\(\lceil (s+1)/2\rceil-1\)}\exp(c_1\log H/\log\log H)
\end{equation}
and
\begin{equation} \label{M*nKm}
L^*_{n,K,s}(H)<H^{2d(n-1)-d\(\lceil (s+1)/2\rceil-1\)}\exp(c_2\log H/\log\log H).
\end{equation}

\end{prop}

In Section~\ref{sec:low}, we show that when $K=\Q$ and $s=n-1$ (\ref{MnKm}) cannot be improved by much; see Theorem~\ref{thm:MnQm}. In particular, it does not hold with $\exp(c_1\log H/\log\log H)$ replaced by a quantity which is $o((\log H)^{(k-1)^2})$, where $n=2k$.

\subsection{Counting vectors of fixed degree}

Let $d$ be a positive integer, and let $\cA_d(H)$, respectively $\cA^*_d(H)$, be the set of algebraic integers of degree $d$ (over $\Q$), respectively algebraic numbers of degree $d$, of height at most $H$. We set
$$
A_d(H) = \card{\cA_d(H)} \mand  A^*_d(H) = \card{\cA^*_d(H)}.
$$

 Put
$$ 
C_5(d)=d2^d\prod^{\lfloor(d-1)/2\rfloor}_{j=1}\frac{d(2j)^{d-2j-1}}{(2j+1)^{d-2j}}
$$
and
$$ 
C_6(d)=\frac{d2^d}{\zeta(d+1)}\prod^{\lfloor(d-1)/2\rfloor}_{j=1}\frac{(d+1)(2j)^{d-2j}}{(2j+1)^{d-2j+1}}.
$$
It follows from the work of Barroero~\cite[Theorem~1.1]{Barroero1} that (see also~\cite[Equation~(1.2)]{Barroero1} for a previous estimate with a weaker error term which follows from~\cite[Theorem~6]{Chern})
\begin{equation} \label{Ad}
A_d(H)=C_5(d)H^{d^2}+O\( H^{d(d-1)}(\log H)^{\rho(d)} \),
\end{equation}
where $\rho(2)=1$ and $\rho(d)=0$ for any $d\ne 2$.

Further,  Masser and Vaaler~\cite[Equation~(7)]{Masser1}  have shown that (see also~\cite[Equation~(1.5)]{Masser2})
\begin{equation} \label{A*d}
A^*_d(H)=C_6(d)H^{d(d+1)}+O\(H^{d^2}(\log H)^{\vartheta(d)}\),
\end{equation}
where $\vartheta(1)=\vartheta(2)=1$ and $\vartheta(d)=0$ for any $d\geq 3$.

For any positive integer $n$, we denote by $M_{n,d}(H)$ the number of multiplicatively dependent $n$-tuples whose coordinates are algebraic integers in $\cA_d(H)$, and we denote by $M^*_{n,d}(H)$ the number of multiplicatively dependent $n$-tuples whose coordinates are algebraic 
numbers in $\cA_d^*(H)$.

For each positive integer $d$, we define $w_0(d)$ to be the number of roots of unity of degree $d$. Let $\varphi$ denote Euler's totient function. Since $\varphi(k)\gg k/\log\log k$ for any integer $k\ge 3$, it follows that
\begin{equation} \label{eq:w0w1}
w_0(d)\ll d^2 \log\log d ,
\end{equation}
 where $d\ge 3$ and the implied constant is absolute.  
 We remark that $w_0(d)$ can be zero, such as for an odd integer $d>1$. 

Given positive integers $n$ and $d$, we define $C_7(n,d)$ and $C_8(n,d)$ as 
$$
C_7(n,d)=\(nw_0(d)+n(n-1)\)C_5(d)^{n-1}
$$
and
$$
C_8(n,d)=\(nw_0(d)+2n(n-1)\)C_6(d)^{n-1}.
$$

\begin{thm} 
\label{thm:Mnd}
Let $d$ and $n$ be positive integers with $n\geq 2$. Then, the following hold. 
\begin{itemize}

\item[(i)]  We have 
\begin{equation} \label{Mnd}
M_{n,d}(H) = C_7(n,d)H^{d^2(n-1)} 
 + O\(H^{d^2(n-1)-d/2}\);
\end{equation} 
 furthermore if $d=2$ or $d$ is odd, we have 
\begin{equation} \label{Mnd'}
\begin{split}
M_{n,d}(H) = C_7&(n,d)H^{d^2(n-1)} \\
& + O\(H^{d^2(n-1)-d}\exp(c_0\log H/\log\log H)\)
\end{split}
\end{equation} 
and 
\begin{equation} \label{M2d}
M_{2,d}(H) = C_7(2,d)H^{d^2} + O\(H^{d^2-d}(\log H)^{\rho(d)}\), 
\end{equation}
where $c_0$ is a positive number which depends only on $n$ and $d$, and $\rho(d)$ has been defined in~\eqref{Ad}.  

\item[(ii)] 
We have 
\begin{equation} \label{M*nd}
 M^*_{n,d}(H) = C_8(n,d)H^{d(d+1)(n-1)}+O\(H^{d(d+1)(n-1)-d/2}\log H\); 
\end{equation} 
 furthermore if $d=2$ or $d$ is odd, we have 
\begin{equation} \label{M*nd'}
\begin{split}
 M^*_{n,d}(H) = C_8&(n,d)H^{d(d+1)(n-1)} \\ 
 & +O\(H^{d(d+1)(n-1)-d}\exp(c\log H/\log\log H)\)
 \end{split}
\end{equation} 
and 
\begin{equation} \label{M*2d}
M^*_{2,d}(H) = C_8(2,d)H^{d(d+1)} + O\(H^{d^2}(\log H)^{\vartheta(d)}\), 
\end{equation}
where $c$ is a positive number which depends only on $n$ and $d$, and $\vartheta(d)$ is defined in~\eqref{A*d}. 
\end{itemize}
\end{thm}

We remark that the case when $d=1$ actually has been included in 
Theorems~\ref{thm:MnK} 
and~\ref{thm:MnK2}.
However, in this case the error term in~\eqref{Mnd'} is 
$H^{n-2+o(1)}$, which is better than that in~\eqref{MnQ} taken with $d=1$. 

The strategy to prove Theorem~\ref{thm:Mnd} is similar to that in proving Theorems~\ref{thm:MnK} and~\ref{thm:MnK2}. 
For each integer $s$ with $0\leq s\leq n-1$, we define $M_{n,d,s}(H)$ and $M^*_{n,d,s}(H)$ to be the number of multiplicatively dependent $n$-tuples of multiplicative rank $s$  whose coordinates are algebraic integers in $\cA_d(H)$ and algebraic numbers in $\cA^*_d(H)$ respectively. Just as in~\eqref{MnK=} we have
\begin{equation} \label{Mnd=}
\left\{ \begin{array}{ll}
                 M_{n,d}(H)=M_{n,d,0}(H)+\cdots+M_{n,d,n-1}(H)\\
                 \\
                 M^*_{n,d}(H)=M^*_{n,d,0}(H)+\cdots+M^*_{n,d,n-1}(H).
                 \end{array} \right.
\end{equation}
For the proof of Theorem~\ref{thm:Mnd}, we make use of~\eqref{Mnd=} and the following result.

\begin{prop} 
\label{prop:Mndm}
Let $d$, $n$   and $s$ be  integers with $d \ge 1$,    $n\geq 2$ and $0 \le s \leq n-1$.  
Then, there exist positive numbers $c_1$ and $c_2$, which depend on $n$ and $d$, such that
\begin{equation} \label{Mndm}
M_{n,d,s}(H)<H^{d^2(n-1)-d(\lceil (s+1)/2\rceil-1)}\exp(c_1\log H/\log\log H)
\end{equation}
and
\begin{equation} \label{M*ndm}
\begin{split}
M^*_{n,d,s}(H)&<H^{d(d+1)(n-1)-d(\lceil (s+1)/2 \rceil-1)}\\
&\qquad \qquad  \qquad  \exp(c_2\log H/\log\log H).
\end{split}
\end{equation}
\end{prop}

We remark that the estimate~\eqref{Mndm} yields an improvement on the upper bound of $H^{d^2(n-1)}$ and~\eqref{M*ndm} yields an improvement of the upper bound $H^{d(d+1)(n-1)}$ for $s$ at least $2$.

\section{Preliminaries} %Section 2

\subsection{Weil height}

We first record a well-known result about the absolute Weil height; see~\cite[Chapter 3]{Lang}.

\begin{lem} 
Let $\alpha$ be a non-zero algebraic number, and let $k$ be an integer. Then
$$
\wH(\alpha^k)=\wH(\alpha)^{|k|}.
$$
\end{lem}

\begin{proof}
This follows from the product formula and the fact that
$$
\wH(\alpha)=\prod_v\max\{1,|\alpha|_v\},
$$
where the product is taken over all inequivalent valuations $v$ appropriately normalized, 
see for example~\cite[Chapter~3, \S1]{Lang}.
\end{proof}

Next we need a result that allows us to compare the naive height $\wH_0$ and the absolute Weil height $\wH$.

\begin{lem} 
\label{lem:H0}
Let $\alpha$  be an algebraic number of degree $d$. Then
$$
\wH_0(\alpha)\leq \(2\wH(\alpha)\)^d.
$$
\end{lem}

\begin{proof}
This follows from noticing that the coefficients of the minimal polynomial $f$ of $\alpha$ can be expressed in terms of elementary symmetric polynomials in the roots of $f$; see for example~\cite[Equation~(6)]{Mahler}.
\end{proof}

For the proofs of Theorems~\ref{thm:MnK} and~\ref{thm:MnK2}, we also need the following result.

\begin{lem} 
\label{lem:Ha}
Let $\alpha$ be an algebraic number of degree $d$, and let $a$ be the leading coefficient of the minimal polynomial of $\alpha$ over the integers. Then
$$
\wH(a\alpha)\leq 2^{d-1}\wH(\alpha)^d.
$$
\end{lem}

\begin{proof}
By definition, we have
$$
\wH(\alpha)=\(a\prod^d_{i=1}\max\{1,|\alpha_i|\}\)^{1/d},
$$
where $\alpha_1,\ldots,\alpha_d$ are the roots of the minimal polynomial of $\alpha$. Then, $a\alpha$ is an algebraic integer, and
$$
\wH(a\alpha)=\(\prod^d_{i=1}\max\{1,|a\alpha_i|\}\)^{1/d}.
$$
Thus
\begin{align*}
\wH(a\alpha)^d & \leq a^d\prod^d_{i=1}\max\{1,|\alpha_i|\}
=a^{d-1}\wH(\alpha)^d,
\end{align*}
which, together with Lemma~\ref{lem:H0}, implies that
$$
\wH(a\alpha)^d\leq \(2\wH(\alpha)\)^{d(d-1)}\wH(\alpha)^d=2^{d(d-1)}\wH(\alpha)^{d^2},
$$
and so
$$
\wH(a\alpha)\leq 2^{d-1}\wH(\alpha)^d
$$
as required.
\end{proof}

\subsection{Multiplicative structure of algebraic  numbers}

Let $K$ be a number field, and let $H$ be a positive real number. We denote by $U_K(H)$ the number of units in the ring of algebraic integers of $K$ of height at most $H$.

\begin{lem} 
\label{lem:units}
 Let $K$ be a number field, and let $r$ be the rank of the unit group as defined in Section~\ref{sec:con}. Then, there exists a positive number $c$, depending on $K$, such that
$$
U_K(H)<c(\log H)^r.
$$
\end{lem}

\begin{proof}
This is~\cite[Part~(ii) of Theorem~5.2 of Chapter~3]{Lang}.
\end{proof}

The next result shows that if algebraic numbers $\alpha_1,\ldots,\alpha_n$ are multiplicatively dependent, then we can find a relation as~\eqref{eq:MultDep}, where the exponents are  not too large. Such a result has  found application in transcendence theory, see for example~\cite{Baker,Matveev,vdPL,Stark}.

\begin{lem} 
\label{lem:exponent}
Let $n\geq 2$, and let $\alpha_1,\ldots,\alpha_n$ be multiplicatively dependent non-zero algebraic numbers of degree at most $d$ and height at most $H$. Then, there is a positive number $c$, which depends only on $n$ and $d$, and there are rational integers $k_1,\ldots,k_n$, not all zero, such that
$$
\alpha^{k_1}_1\cdots\alpha^{k_n}_n=1
$$
and
$$
\max_{1\leq i\leq n}|k_i|<c(\log H)^{n-1}.
$$
\end{lem}

\begin{proof}
This follows from~\cite[Theorem~1]{vdPL}. For an explicit constant $c$, we refer to~\cite[Corollary~3.2]{LM}.
\end{proof}

Let $x$ and $y$ be positive real numbers with $y$ larger than 2, and let $\psi(x,y)$ denote the number of positive integers not exceeding $x$ which contain no prime factors greater than $y$. Put
$$
Z=\(\log\(1+\frac{y}{\log x}\)\)\frac{\log x}{\log y}+\(\log\(1+\frac{\log x}{y}\)\)\frac{y}{\log y}
$$
and
$$
u=(\log x)/(\log y).
$$

\begin{lem} 
\label{lem:psixy}
For $2<y\leq x$, we have
\begin{align*}
\psi(x&,y)\\
&=\exp\( Z \(1+O((\log y)^{-1})+O((\log\log x)^{-1})+O((u+1)^{-1}) \) \).
\end{align*}
\end{lem}

\begin{proof}
This is~\cite[Theorem~1]{dB}.
\end{proof}

\subsection{Counting special algebraic numbers} 

In this section, we count two special kinds of algebraic numbers.

\begin{lem} 
\label{lem:coeff}
Let $K$ be a number field of degree $d$, and let $u$ and $v$ be non-zero integers with $u>0$. Then, there is a positive number $c$, which depends on $K$, such that the number of elements $\alpha$ in $K$ of height at most $H$, whose minimal polynomial has leading coefficient $u$ and constant coefficient $v$, is at most
$$
\exp(c\log H/\log\log H).
$$
\end{lem}
\begin{proof}
Let $c_1,c_2,\ldots$ denote positive numbers depending on $K$. 
Let $N_{K/\Q}$ be the norm function from $K$ to $\Q$. 
Suppose that $\alpha$ is an element of $K$ of height at most $H$ whose minimal polynomial has leading coefficient $u$ and constant coefficient $v$. Then, we see that $u\alpha$ is an algebraic integer in $K$, and 
$$
N_{K/\Q}(\alpha) = (-1)^d v/u \mand N_{K/\Q}(u\alpha) = (-1)^d u^{d-1}v. 
$$
By Lemma~\ref{lem:Ha}, we further have $\wH(u\alpha) \le 2^{d-1}H^d$. 
Note that $u$ is fixed, so the number of such $\alpha$ does not exceed the number of algebraic integers $\beta \in K$ of height at most $2^{d-1}H^d$ and satisfying 
\begin{equation}
\label{eq:norm eq}
N_{K/\Q}(\beta) = (-1)^d u^{d-1}v. 
\end{equation}

We say that two algebraic integers $\beta_1$ and $\beta_2$ in $K$ are equivalent if the principal integral ideals $\langle \beta_1 \rangle$ and $\langle \beta_2 \rangle$  are equal.  We note that, 
using~\cite[Chapter~3, 
Equation~(7.8)]{BS}, the number $E$ of equivalence classes of solutions of~\eqref{eq:norm eq} is at most $\tau(|u^{d-1}v|)^d$, where, for any positive integer $k$, $\tau(k)$ denotes the number of positive integers which divide $k$.
By Wigert's Theorem, see~\cite[Theorem~317]{Hardy},
\begin{equation}
\label{eq:E1}
E<\exp \( c_1\log (3|uv|)/\log\log (3|uv|) \).
\end{equation}
Further by Lemma~\ref{lem:H0} $u$ and $v$ are at most $(2H)^d$ in absolute value, hence
\begin{equation} \label{eq:E2}
E<\exp(c_{2}\log H/\log\log H).
\end{equation}

Besides, if two solutions $\beta_1$ and $\beta_2$ of~\eqref{eq:norm eq} are equivalent, then $\beta_1/\beta_2$ is a unit $\eta$ in the ring of algebraic integers of $K$. But
$$
\wH(\eta) \leq \wH(\beta_1)\wH((\beta_2)^{-1}) \leq 2^{2(d-1)}H^{2d}.
$$
By Lemma~\ref{lem:units} the number of such units is at most
\begin{equation}\label{eq:U}
U_K(2^{2(d-1)}H^{2d}) \leq c_{3}( \log H)^r.
\end{equation}
Our result now follows from~\eqref{eq:E2} and~\eqref{eq:U}.
\end{proof}

We remark that if we set $u=1$, then Lemma~\ref{lem:coeff} gives an upper bound for the number of algebraic integers in $K$ of norm $\pm v$ and of height at most $H$.

Given integer $d\ge 1$, let $\cC^*_d(H)$ be the set of algebraic numbers $\alpha$ of degree $d$ and height at 
most $H$ such that $\alpha\eta$ is also of degree $d$ for
 some root of unity $\eta\ne \pm 1$, and let $\cC_d(H)$ be the set of algebraic integers contained 
 in $\cC^*_d(H)$. Here, we want to estimate the sizes of $\cC_d(H)$ and $\cC^*_d(H)$. 

For this we need some preparations.  
Given a polynomial $f=a_dX^d + \cdots + a_1X +a_0\in \Q[X]$ of degree $d$, we call it \textit{degenerate} if it has two distinct roots whose quotient is a root of unity. Besides, we define its \textit{height} as 
$$
\wH(f) = \max\{|a_d|,\ldots,|a_1|,|a_0|\}, 
$$
and we denote by $G_f$ the \textit{Galois group} of the splitting field of $f$ over $\Q$. 
Let $S_d$ be the full symmetric group of $d$ symbols. 

Define 
$$
\cE_d(H) =\{\textrm{monic $f\in \Z[X]$ of degree $d$: $\wH(f)\le H$ and $G_f\ne S_d$} \}
$$
and 
$$
\cE^*_d(H) =\{\textrm{$f\in \Z[X]$ of degree $d$: $\wH(f)\le H$ and $G_f\ne S_d$} \}. 
$$
The study of the sizes of $\cE_d(H)$ and $\cE^*_d(H)$ was 
initiated by van der Waerden~\cite{Waerden}. 
 Here, we recall a recent result due to Dietmann~\cite[Theorem~1]{Dietmann2013}:   
\begin{equation} \label{eq:Galois1}
|\cE_d(H) | \ll H^{d-1/2}. 
\end{equation}
Besides, by a result of Cohen~\cite[Theorem~1]{Cohen} (taking $K=\Q, s=n+1$ and $r=1$ there), we directly have 
\begin{equation} \label{eq:Galois2}
|\cE^*_d(H) | \ll H^{d+1/2}\log H. 
\end{equation}

We also put
$$
\cF_d(H) = \{\textrm{monic $f\in \Z[X]$ of degree $d:$\, $\wH(f)\le H$, $f$ is degenerate} \}
$$
and 
$$
\cF^*_d(H) = \{\textrm{$f\in \Z[X]$ of degree $d:$\, $\wH(f)\le H$, $f$ is degenerate} \}. 
$$
Applying~\cite[Theorems~1 and~4]{DS}, we have 
\begin{equation}
\label{eq:degenerate}
|\cF_d(H)| \ll H^{d-1} \mand |\cF^*_d(H)| \ll H^d. 
\end{equation}

We are now ready to prove the following lemma. 

\begin{lem} \label{lem:special}
We have: 
\begin{itemize}
\item[(i)] for any integer $d\ge 1$, 
$$
|\cC_d(H)| \ll H^{d(d-1/2)} \quad \text{and} \quad  |\cC^*_d(H)|\ll H^{d(d+1/2)}\log H;
$$ 
\item[(ii)]  for $d=2$ or for $d$  odd,  
$$
|\cC_d(H)| \ll H^{d(d-1)} \mand |\cC^*_d(H)|\ll H^{d^2}. 
$$ 
\end{itemize}
\end{lem}

\begin{proof}
Pick an arbitrary element $\alpha \in \cC_d(H)$. We let $f$ be its minimal polynomial over $\Z$, and let the $d$ roots of $f$
be $\alpha_1,\ldots,\alpha_d$ with $\alpha_1=\alpha$. Since $\alpha$ is of height at most $H$, 
by Lemma~\ref{lem:H0} we have 
$$
\wH(f)\le (2H)^d.
$$ 
 
By definition, there is a root of unity $\eta \ne \pm 1$ such that $\alpha\eta$ is also of degree $d$. 
If $\eta \in \Q(\alpha)$, then under an isomorphism sending $\alpha$ to $\alpha_i$, $\eta$ is mapped to one of its conjugates $\eta_i$ in $\Q(\alpha_i)$, which implies that $\eta \in \Q(\alpha_i)$ for any $1\le i \le d$. 
Indeed, the image $\eta_i$ of $\eta$ in $\Q(\alpha_i)$ multiplicatively generates the same group 
as $\eta$, and thus $\eta$ is a power of $\eta_i$, so $\eta \in \Q(\alpha_i)$. 
Hence,  $\bigcap_{i=1}^{d}\Q(\alpha_i) \ne \Q$, then we must have $G_f \ne S_d$,
that is, 
\begin{equation}
\label{eq:f in E}
f \in  \cE_d((2H)^d).
\end{equation}
Furthermore, since $f$ is irreducible, in this case $d \ne 2$. 
We also note that since $\eta$ is of even degree  $\varphi(k)$, where $k > 2$ is the smallest 
positive integer with $\eta^k = 1$, 
this case does not happen when $d$ is odd. 

Now, we assume that $\eta \not\in \Q(\alpha)$. 
Let $K=\Q(\eta,\alpha_1,\ldots,\alpha_d)$, and let $G$ be the Galois group $\Gal(K/\Q)$, where $K$ is indeed a Galois extension over $\Q$. We construct a disjoint union $G= \bigcup_{i=1}^{d}G_i$, where 
$$
G_i = \{\phi\in G: \, \phi(\alpha)=\alpha_i \}. 
$$
So, for each $1\le i \le d$ 
$$
G_i \alpha\eta = \{\phi(\alpha\eta):\, \phi \in G_i \}
=\{\alpha_i\phi(\eta):\, \phi \in G_i \}.
$$ 
Since  $\alpha\eta$ is of degree $d$, we have 
\begin{equation}
\label{eq: Union Card}
  \left | \bigcup_{i=1}^{d}G_i \alpha\eta \right| = d.
\end{equation}
Note that $\alpha_1=\alpha$, then $G_1 = \Gal(K/\Q(\alpha))$. Since $\eta \not\in \Q(\alpha)$, 
there exist two morphisms $\phi_1,\phi_2 \in G_1$ such that $\phi_1(\eta) \ne \phi_2(\eta)$. 
That is, $|G_1 \alpha\eta| \ge 2$. 
Trivially, $|G_i \alpha\eta| \ge 1$ for $2\le i \le d$. 
We now see from~\eqref{eq: Union Card} that there are two distinct indices $i, j$ such that 
$G_i\alpha\eta \cap G_j\alpha\eta \ne \emptyset$, 
which implies that $\alpha_i/\alpha_j$ is a root of unity and thus $f$ is degenerate, that is, 
\begin{equation}
\label{eq:f in F}
f \in  \cF_d((2H)^d).
\end{equation}

Hence, if $\alpha \in \cC_d(H)$, then  combing~\eqref{eq:f in E} and~\eqref{eq:f in F} 
with~\eqref{eq:Galois1} and~\eqref{eq:degenerate}, respectively, we derive the first inequality in~(i).
 If $d=2$ or $d$ is odd, by the above discussion we always have~\eqref{eq:f in F}, and thus the first inequality in~(ii) follows from~\eqref{eq:degenerate}. 
Similar arguments also apply to estimate $|\cC^*_d(H)|$ by using~\eqref{eq:Galois2} and~\eqref{eq:degenerate}. 
\end{proof}

\section{Proofs of Propositions~\ref{prop:MnKm} and~\ref{prop:Mndm}} %Section 3

\subsection{Proof of Proposition~\ref{prop:MnKm}} 

Let $c_3,c_4,\ldots$ denote positive numbers  depending on $n$ and $K$. 
Let 
$\pmb{\nu}=(\nu_1,\ldots,\nu_n)$ be a multiplicatively dependent vector of multiplicative rank $s$ whose coordinates are from $K$ and have height at most $H$. 
Set $m = s+1$. Then, there are $m$ distinct integers $j_1,\ldots,j_m$ from $\{1,\ldots,n\}$ for which $\nu_{j_1},\ldots,\nu_{j_m}$ are multiplicatively dependent and there are non-zero integers $k_{j_1},\ldots,k_{j_m}$ for which
\begin{equation} \label{eq:mult}
\nu^{k_{j_1}}_{j_1}\cdots \nu^{k_{j_m}}_{j_m}=1,
\end{equation}
and further by Lemma~\ref{lem:exponent}, we can assume that
\begin{equation} \label{eq:expo}
\max\{|k_{j_1}|,\ldots,|k_{j_m}|\}<c_3(\log H)^{m-1}.
\end{equation}
Let $P$ be the set of indices $i$ for which $k_i$ is positive, and let $N$ be the set of indices $i$ for which $k_i$ is negative. Then
\begin{equation} \label{eq:sep}
\prod_{i\in P}\nu^{k_i}_{i}=\prod_{i\in N}\nu^{-k_i}_{i}.
\end{equation}
Plainly, either $\card{P}$ or $\card{N}$ is at least $\lceil m/2\rceil$. 

Let $I=\{j_1,\ldots,j_m\}$, and let $I_0$ be the subset of $I$ consisting of the indices $i$ for which $k_i$ is positive if 
$\card{P}\geq\lceil m/2 \rceil$, and otherwise let $I_0$ be the subset of $I$ consisting of the indices $i$ for which $k_i$ is negative. 
Note that 
\begin{equation} \label{eq:I0}
\card{ I_0}\geq\left\lceil\frac{m}{2}\right\rceil.
\end{equation}
It follows from~\eqref{eq:sep} that
\begin{equation} \label{eq:maineq}
\prod_{i\in I_0}\nu^{|k_i|}_i=\prod_{i\in I\backslash I_0}\nu^{|k_i|}_i.
\end{equation}

For each coordinate $\nu_i$, $i\in I$, let $a_i$ be the leading coefficient of the minimal polynomial of $\nu_i$ over the integers. Note that $a_i\nu_i$ is an algebraic integer, and we can rewrite~\eqref{eq:maineq} as
\begin{equation} \label{eq:maineq1}
\prod_{i\in I_0}(a_i\nu_i)^{|k_i|}=\prod_{i\in I_0}a^{|k_i|}_i\prod_{i\in I\backslash I_0}\nu_i^{|k_i|}.
\end{equation}

We  first establish~\eqref{MnKm}. Accordingly, we fix non-zero algebraic integers $\nu_i \in \cB_K(H)$ for $i$ from $\{1,\ldots,n\}\backslash I_0$ and estimate the number of solutions of~\eqref{eq:maineq} in algebraic integers $\nu_i$, $i\in I_0$, from $\cB_K(H)$. Observe that the number of cases when we consider an equation of the form~\eqref{eq:maineq} is, by~\eqref{eq:expo}, at most
$$
\binom{n}{m}\(2c_3(\log H)^{(m-1)}\)^m B_K(H)^{n-\card{ I_0}},
$$
and, by~\eqref{BK} and~\eqref{eq:I0}, is at most
\begin{equation} \label{eq:cases1}
c_4H^{d(n-\lceil m/2 \rceil)}(\log H)^{c_5}.
\end{equation}

Let $q_1,\ldots,q_t$ be the primes which divide
$$
\prod_{i\in I\backslash I_0}N_{K/\Q}(\nu_i),
$$
where $N_{K/\Q}$ is the norm  from $K$ to $\Q$.
Since the height of $\nu_i$ is at most $H$, it follows from Lemma~\ref{lem:H0} that
\begin{equation} \label{eq:normv}
|N_{K/\Q}(\nu_i)|\leq (2H)^d, \quad i=1,2,\ldots,n,
\end{equation}
and since $\card {I\backslash I_0}\leq n$, we see that
\begin{equation} \label{eq:normv2}
\left|\prod_{i\in I\backslash I_0}N_{K/\Q}(\nu_i)\right|\leq (2H)^{dn}.
\end{equation}
Let $p_1,\ldots,p_k$ be the first $k$ primes, where $k$ satisfies
$$
p_1\cdots p_k\leq\left|\prod_{i\in I\backslash I_0}N_{K/\Q}(\nu_i)\right|<p_1\cdots p_{k+1}.
$$
Let $T$ denote the number of positive integers up to $(2H)^d$ which are composed only of primes from $\{q_1,\ldots,q_t\}$. We see that $T$ is bounded from above by the number of positive integers up to $(2H)^d$ which are composed of primes from $\{p_1,\ldots,p_k\}$. By~\eqref{eq:normv2}, we obtain
$$
\sum_{\textrm{prime $p\le p_k$}} \log p \ll \log H,
$$
which, combined with the prime number theorem, yields
$$
p_k<c_6\log H.
$$
Therefore we have
$$
T\leq \psi\((2H)^d, c_6\log H \),
$$
and thus by Lemma~\ref{lem:psixy},
\begin{equation} \label{eq:T1}
T<\exp(c_7\log H/\log\log H).
\end{equation}
It follows that if $(\nu_i,i\in I_0)$ is a solution of~\eqref{eq:maineq}, then $|N_{K/\Q}(\nu_i)|$ is composed only of primes from $\{q_1,\ldots,q_t\}$, and so $N_{K/\Q}(\nu_i)$ is one of at most $2T$ integers of absolute value at most $(2H)^d$. Let $a$ be one of those integers.

By Lemma~\ref{lem:coeff}, the number of algebraic integers $\alpha$ from $K$ of height at most $H$ for which
\begin{equation} \label{eq:normeq}
N_{K/\Q}(\alpha)=a  
\end{equation}
is at most $\exp(c_8\log H/\log \log H)$.
Therefore, by~\eqref{eq:T1},  and~\eqref{eq:normeq}, the number of $\card{ I_0}$-tuples $(\nu_i, i\in I_0)$ which give a solution of~\eqref{eq:maineq} is at most $\exp(c_{9}\log H/\log\log H)$. Recalling $m = s+1$, we see that our bound~\eqref{MnKm} now 
follows from~\eqref{eq:cases1}.

We  now establish~\eqref{M*nKm}. We first remark by Lemmas~\ref{lem:H0} and~\ref{lem:Ha} that
\begin{equation} \label{eq:ai}
0< a_i \leq (2H)^d
\end{equation}
and
\begin{equation} \label{eq:ainu}
\wH(a_i\nu_i)\leq 2^{d-1}H^d,
\end{equation}
for $i=1,\ldots,n$. 
Moreover, without loss of generality we can assume that $I \setminus I_0$ is not empty. Indeed, if $I \setminus I_0$ is empty, then we can replace an arbitrary coordinate $\nu_i,i\in I$, by its inverse $\nu_i^{-1}$. 

In view of~\eqref{eq:maineq1}, we proceed by fixing $a_i$ for $i$ in $I_0$ and $\nu_i$ for $i$ in $\{1,\ldots,n\}\backslash I$. Since  
$I\backslash I_0$ is non-empty, say that it contains $i_1$. We further fix $\nu_i$ for $i$ in $I\backslash I_0$ with $i\neq i_1$, and then  
the corresponding leading coefficient $a_i$ is also  fixed. Let  
$$
\beta=\prod_{i\in I_0}a^{|k_i|}_i\prod_{\substack{i\in I\backslash I_0 \\ i \ne i_1}}(a_i\nu_i)^{|k_i|},
$$
which is actually a fixed non-zero algebraic integer, then $N_{K/\Q}(\beta)$ is a fixed non-zero integer.   
Note that the left-hand side of~\eqref{eq:maineq1} is an algebraic integer, so $\beta \nu_{i_1}$ is an algebraic integer, and then $N_{K/\Q}(\beta\nu_{i_1})$ is also an algebraic integer. Thus, the leading coefficient $a_{i_1}$ divides $N_{K/\Q}(\beta)$. It follows that the prime factors of $a_{i_1}$ divide 
$$
\prod_{i\in I_0}a_i\prod_{\substack{i\in I\backslash I_0 \\ i \ne i_1}}N_{K/\Q}(a_i\nu_i). 
$$
Since the heights of  $\nu_1, \ldots, \nu_n$ are at most $H$, we see, as in 
the proof of the estimate~\eqref{eq:T1},
that there are at most $\exp(c_{10}\log H/\log\log H)$ possibilities for the leading coefficient $a_{i_1}$. Note that by Lemma~\ref{lem:H0} there are at most $2(2H)^d$ possibilities for the constant coefficient of the minimal polynomial of $\nu_{i_1}$. 
 Thus, by Lemma~\ref{lem:coeff}, there are at most
\begin{equation} \label{eq:22}
H^d\exp(c_{11} \log H/\log \log H)
\end{equation}
possible values  of $\nu_{i_1}$ that we need to consider. In total we have, by~\eqref{B*K}, \eqref{eq:ai} 
and~\eqref{eq:22}, at most
\begin{align*}
\binom{n}{m}\(2c_3(\log H)^{(m-1)}\)^m (2H)^{d\card{ I_0}} & H^{2d(n-\card{ I_0}-1)}H^d \\
& \exp(c_{11} \log H/ \log \log H)
\end{align*}
equations of the form~\eqref{eq:maineq1}. Since $\card{ I_0}\geq\lceil \frac{m}{2}\rceil$, the number of such equations is at most
\begin{equation} \label{eq:total}
H^{2dn-d(\lceil \frac{m}{2}\rceil +1)}\exp(c_{12}\log H/ \log \log H).
\end{equation}
%On the other hand, if $I\backslash I_0$ is empty, then $\card{I_0}=\card{I}=m$ and the number of such equations is $O(H^{dm})$, so \eqref{eq:total} again holds.

 Let us put
\begin{equation} \label{eq:gamma0}
\gamma_0=\prod_{i\in I_0}a_i^{|k_i|}\prod_{i\in I\backslash I_0}(a_i\nu_i)^{|k_i|}
\end{equation}
and
$$
\gamma_1=\prod_{i\in I\backslash I_0}a^{|k_i|}_i.
$$
Notice that once $\nu_i$ is fixed for $i$ in $I\backslash I_0$, so is $a_i$ and thus $\gamma_1$ is fixed. 
Then,~\eqref{eq:maineq1} can be rewritten as
\begin{equation} \label{eq:maineq2}
\gamma_1\prod_{i\in I_0}(a_i\nu_i)^{|k_i|}=\gamma_0,
\end{equation}
and we seek an estimate for the number of solutions of~\eqref{eq:maineq2} in algebraic numbers $\nu_i$ from $\cB^*_K(H)$ with leading coefficient $a_i$ for $i \in I_0$. 

Note that $\gamma_0$ is an algebraic integer and $\gamma_1$ is an integer. Let $q_1,\ldots,q_t$ be the prime factors of
$$
\prod_{i\in I_0}a_i\prod_{i\in I\backslash I_0}N_{K/\Q}(a_i\nu_i).
$$
Then, by~\eqref{eq:gamma0} and~\eqref{eq:maineq2}, for each index $i\in I_0$ the prime factors of $N_{K/\Q}(a_i\nu_i)$ are from $\{q_1,\ldots, q_t\}$. It follows from~\eqref{eq:ai}, \eqref{eq:ainu} and Lemma~\ref{lem:H0}  that
$$
\left|\prod_{i\in I_0}a_i\prod_{i\in I\backslash I_0}N_{K/\Q}(a_i\nu_i)\right|  \leq (2H)^{d\card{ I_0}}(2^dH^d)^{d\card{I\backslash I_0}}  \leq (2H)^{d^2n}.
$$

We can now argue as in our proof of~\eqref{MnKm} that the number of solutions of~\eqref{eq:maineq2} in algebraic integers $a_i\nu_i$, $i\in I_0$, from $K$ of height at most $2^{d-1}H^d$ is at most $\exp(c_{13}\log H/\log\log H)$. The result~\eqref{M*nKm} now follows from~\eqref{eq:total}.

\subsection{Proof of Proposition~\ref{prop:Mndm}} 

Let $c_3,c_4,\ldots$ denote positive numbers depending on  $n$ and $d$. 
Notice that if $\pmb{\nu}=(\nu_1,\ldots,\nu_n)$ is a multiplicatively dependent vector of multiplicative rank $s$ whose coordinates are from $\cA^*_d(H)$.
Set $m = s+1$. Then, there are $m$ distinct integers $j_1,\ldots,j_m$ from $\{1,\ldots,n\}$ for which $\nu_{j_1},\ldots,\nu_{j_m}$ are multiplicatively dependent and there are non-zero integers $k_{j_1},\ldots,k_{j_m}$ for which~\eqref{eq:mult} holds,  and by Lemma~\ref{lem:exponent}, we can suppose that~\eqref{eq:expo} holds. Let $I=\{j_1,\ldots,j_m\}$ and $I_0$ be defined as in the proof of 
Proposition~\ref{prop:MnKm}, so that~\eqref{eq:I0}  and~\eqref{eq:maineq} hold.

We first establish~\eqref{Mndm}. Fixing non-zero algebraic integers $\nu_i \in \cA_d(H)$ for $i \in \{1,\ldots,n\}\backslash I_0$, we want to estimate the number of solutions of~\eqref{eq:maineq} in algebraic integers $\nu_i \in \cA_d(H)$ for $i\in I_0$. The number of cases when we consider an equation of the form~\eqref{eq:maineq} is, by~\eqref{eq:expo}, at most
$$
\binom{n}{m}\(2c_3(\log H)^{m-1}\)^{m}A_d(H)^{n-\card{I_0}},
$$
which, by~\eqref{Ad}, is at most
\begin{equation} \label{eq:cases4}
c_4H^{d^2(n-\card{I_0})}(\log H)^{m(m-1)}.
\end{equation}

For each $i\in I_0$, by~\eqref{eq:maineq} the prime factors of $N_{\Q(\nu_i)/\Q}(\nu_i)$ divide 
$$
\prod_{j\in I \setminus I_0} N_{\Q(\nu_j)/\Q}(\nu_j). 
$$  
Just as in the proof of Proposition~\ref{prop:MnKm}, we can apply Lemma~\ref{lem:H0} and Lemma~\ref{lem:psixy} to conclude that, for $i\in I_0$, $N_{\Q(\nu_i)/\Q}(\nu_i)$ is one of at most $T$ integers, where, as in~\eqref{eq:T1},
$$
T<\exp(c_5\log H/\log\log H).
$$
Then, estimating the number of possible choices of the minimal polynomial of $\nu_i$ over the integers by using Lemma~\ref{lem:H0},  we see that there are at most
\begin{equation} \label{eq:viI0}
d\(2(2H)^d+1\)^{d-1}\exp(c_5\log H/\log\log H)
\end{equation}
possible values of each $\nu_i$ for $i \in I_0$.
We now fix $\card{I_0}-1$ of the terms $\nu_i$ with $i$ in $I_0$. Let $i_0 \in I_0$ denote the index of the term which is not fixed. Then, $\nu_{i_0}$ is a solution of
\begin{equation} \label{eq:subeq}
x^{|k_{i_0}|}=\eta_0,
\end{equation}
where
$$
\eta_0=\prod_{\substack{i\in I_0 \\  i\neq i_0}}\nu^{-|k_i|}_i\prod_{i\in I\backslash I_0}\nu^{|k_i|}_i.
$$
If $\nu_{i_0}$ and $\mu_{i_0}$ are two solutions of~\eqref{eq:subeq} from $\cA_d(H)$, then $\nu_{i_0}/\mu_{i_0}$ is a $|k_{i_0}|$-th root of unity. But the degree of $\nu_{i_0}/\mu_{i_0}$ is at most $d^2$, and so there are at most $c_6$ possibilities for $\nu_{i_0}/\mu_{i_0}$ when $d$ is fixed. It follows from~\eqref{eq:viI0} that each 
equation~\eqref{eq:maineq} has at most
\begin{equation} \label{eq:solutions1}
H^{d(d-1)(\card{I_0}-1)}\exp(c_7\log H/\log\log H)
\end{equation}
solutions. Thus by~\eqref{eq:cases4} and~\eqref{eq:solutions1}, we have 
\begin{equation} \label{eq:Mndm1}
M_{n,d,s}(H)<H^{d^2(n-\card{I_0})+d(d-1)(\card{I_0}-1)}\exp(c_8\log H/\log\log H).
\end{equation}
Further, by~\eqref{eq:I0},
\begin{equation} \label{eq:dI0}
d^2(n-\card{I_0})+d(d-1)(\card{I_0}-1)\leq d^2(n-1)-d\(\left\lceil\frac{m}{2}\right\rceil-1\).
\end{equation}
Now,~\eqref{Mndm} follows from~\eqref{eq:Mndm1} and~\eqref{eq:dI0}.

We next establish~\eqref{M*ndm}. 
For each $i\in I$, let $a_i$ denote the leading coefficient of the minimal polynomial of $\nu_i$ over the integers. 
Without loss of generality, we can assume that $I \setminus I_0$ is not empty. Indeed, if $I \setminus I_0$ is empty, then we can replace an arbitrary coordinate $\nu_i,i\in I$, by its inverse $\nu_i^{-1}$.

In view of~\eqref{eq:maineq1}, we proceed by first fixing positive integers $a_i$ for $i \in I_0$. Since $I\backslash I_0$ is non-empty, say that it contains $i_1$. We next fix $\nu_i$ for $i$ in $i \in \{1,\ldots,n\}\backslash I_0$ with $i \neq i_1$, and then the corresponding $a_i$ is also fixed.  Let 
$$
\beta=\prod_{i\in I_0}a^{|k_i|}_i\prod_{\substack{i\in I\backslash I_0 \\ i \ne i_1}}(a_i\nu_i)^{|k_i|}, 
$$
which is a fixed non-zero algebraic integer. 
Notice that the left-hand side of~\eqref{eq:maineq1} is an algebraic integer, so $\beta \nu_{i_1}$ is also an algebraic integer, and thus as in the proof of~\eqref{M*nKm} the prime factors of the leading coefficient $a_{i_1}$ divide 
$$
\prod_{i\in I_0}a_i\prod_{\substack{i\in I\backslash I_0 \\ i \ne i_1}}N_{\Q(\nu_i)/\Q}(a_i\nu_i). 
$$
 Since the heights of $\nu_1, \ldots, \nu_n$ are at most $H$ and their degrees are all equal to $d$, we see, as in the proof of~\eqref{eq:T1}, that  there are at most $\exp(c_9 \log H/ \log\log H)$ possibilities for the leading coefficient $a_{i_1}$. 
Then, combining this result with Lemma~\ref{lem:H0}, we know that the number of the possibilities for the minimal polynomial of $\nu_{i_1}$ is at most 
$$
H^{d^2}\exp(c_{10} \log H / \log \log H).
$$
Thus, there are at most
\begin{equation} \label{eq:nu_i1}
H^{d^2}\exp(c_{11} \log H / \log \log H)
\end{equation}
possible values of $\nu_{i_1}$ that we need to  consider. 

Hence, the number of cases of the equation~\eqref{eq:maineq1} to be considered is, by~\eqref{eq:expo}, \eqref{eq:ai} and~\eqref{eq:nu_i1}, at most
\begin{align*}
\binom{n}{m}\(2c_3(\log H)^{m-1}\)^{m}(2H)^{d\card{I_0}}  &A^*_d(H)^{n-\card{I_0}-1}H^{d^2} \\ 
& \exp(c_{11} \log H / \log \log H),
\end{align*}
which, by~\eqref{A*d}, is at most
\begin{equation} \label{eq:cases5}
H^{d(d+1)(n-\card{I_0}-1)+d\card{I_0}+d^2}\exp(c_{12} \log H / \log \log H).
\end{equation}

We  now estimate the number of solutions of~\eqref{eq:maineq1} in algebraic numbers $\nu_i\in \cA^*_d(H)$ for $i \in I_0$  with minimal polynomial having leading coefficient $a_i$.   
It follows from~\eqref{eq:maineq1} that for each $i\in I_0$ the prime factors of $N_{\Q(\nu_i)/\Q}(a_i\nu_i)$ divide
$$
\prod_{j\in I_0}a_j\prod_{j\in I\backslash I_0}N_{\Q(\nu_j)/\Q}(a_j\nu_j).
$$
Thus, by Lemma~\ref{lem:H0}, Lemma~\ref{lem:Ha} and Lemma~\ref{lem:psixy}, as in the proof of~\eqref{eq:T1}, there is a set of at most $T$ integers, where
$$
T<\exp(c_{13}\log H/\log\log H),
$$
and $N_{\Q(\nu_i)/\Q}(a_i\nu_i)$ belongs to that set. Since $a_i$ is fixed, the norm $N_{\Q(\nu_i)/\Q}(\nu_i)$ also belongs to a set of cardinality at most $T$ for $i \in I_0$. 
Notice that for the minimal polynomial of $\nu_i, i
\in I_0$, if $N_{\Q(\nu_i)/\Q}(\nu_i)$ is fixed, then the constant coefficient is also fixed, because  the leading coefficient $a_i$ has already been fixed.  
Hence, counting possible choices of the minimal polynomial of $\nu_i$ by using Lemma~\ref{lem:H0}, we see that there are at most
\begin{equation} \label{eq:viI0*}
H^{d(d-1)}\exp(c_{14}\log H/\log\log H)
\end{equation}
possible values of $\nu_i$ for $i \in I_0$. We now fix $\card{I_0}-1$ of the coordinates $\nu_i$ with $i\in I_0$ and argue as before to conclude from~\eqref{eq:viI0*} that each equation~\eqref{eq:maineq1} has at most
\begin{equation} \label{eq:solutions2}
H^{d(d-1)(\card{I_0}-1)}\exp(c_{15}\log H/\log\log H)
\end{equation}
solutions. Thus, by~\eqref{eq:cases5} and~\eqref{eq:solutions2}, we obtain 
\begin{equation} \label{eq:M*ndm1}
\begin{split}
M^*_{n,d,s}(H) &<H^{d(d+1)(n-\card{I_0}-1)+d\card{I_0}+d^2+d(d-1)(\card{I_0}-1)} \\
& \qquad \qquad \qquad \qquad \qquad \exp(c_{16}\log H/\log\log H).
\end{split}
\end{equation}
Observing that
\begin{align*}
&d(d+1)(n-\card{I_0}-1)+d\card{I_0}+d^2 +d(d-1)(\card{I_0}-1) \\
& \quad =d(d+1)(n-1)-d(\card{I_0}-1),
\end{align*}
our result~\eqref{M*ndm} now follows from~\eqref{eq:I0} and~\eqref{eq:M*ndm1}.

\section{Proof of Main Results} %Section 4

\subsection{Proof of Theorem~\ref{thm:MnK}} 

By~\eqref{MnK=} and~\eqref{MnKm}, there is a positive number $c$ which depends on $n$ and $K$ such that
\begin{equation} \label{eq:MnKH}
\begin{split}
L_{n,K}(H) =L_{n,K,0}(H)&+L_{n,K,1}(H)\\
&+O(H^{d(n-1)-d}\exp(c\log H/\log\log H)).
\end{split}
\end{equation}

Each such vector $\pmb{\nu}$ of multiplicative rank $0$ has an index $i_0$ for which $\nu_{i_0}$ is a root of unity. Accordingly, we have
$$
nw(B_K(H)-w-1)^{n-1}\leq L_{n,K,0}(H)\leq nwB_K(H)^{n-1},
$$
and thus by~\eqref{BK}
\begin{equation} \label{eq:MnK1}
\begin{split}
L_{n,K,0}(H)=nwC_1(K)^{n-1}&H^{d(n-1)}(\log H)^{r(n-1)}\\
& +O\(H^{d(n-1)}(\log H)^{r(n-1)-1}\).
\end{split}
\end{equation}

We next estimate $L_{n,K,1}(H)$. Each such vector $\pmb{\nu}$ of rank $1$ has a pair of indices $(i_0,i_1)$, two coordinates $\nu_{i_0}$ and $\nu_{i_1}$ from $\cB_K(H)$ and non-zero integers $k_{i_0}$ and $k_{i_1}$ such that $\nu^{k_{i_0}}_{i_0}\nu^{k_{i_1}}_{i_1}=1$. There are $n(n-1)/2$ pairs $(i_0,i_1)$.
By Lemma~\ref{lem:exponent}, the number of such vectors associated with two distinct such pairs $(i_0,i_1)$ and $(i_2,i_3)$ is
\begin{equation} \label{eq:twop}
O\(B_K(H)^{n-2} (\log H)^4\).
\end{equation}

We now estimate the number of $n$-tuples $\pmb{\nu}$ whose coordinates are from $\cB_K(H)$ for which
$$
\nu^{k_{i_0}}_{i_0}\nu^{k_{i_1}}_{i_1}=1
$$
with $(k_{i_0},k_{i_1})$ equal to $(t,t)$ or $(t,-t)$ for some non-zero integer $t$. We have $(B_K(H)-w-1)^{n-2}$ choices for the coordinates of $\pmb{\nu}$ associated with indices different from $i_0$ and $i_1$, because they are non-zero and not roots of unity. Also there are $B_K(H)-w-1$ choices for the $i_0$-th coordinate, and once it is determined, say $\nu_{i_0}$, then the $i_1$-th coordinate is of the form $\eta\nu_{i_0}$ or $\eta\nu^{-1}_{i_0}$, where $\eta$ is a root of unity from $K$. Note that
$$
\wH(\eta\nu_{i_0})=\wH(\nu_{i_0})=\wH(\eta\nu^{-1}_{i_0}),
$$
and that $\eta\nu^{-1}_{i_0}$ is only counted when  $\nu_{i_0}$ is a unit in the ring of algebraic integers of $K$. Thus, we have
\begin{equation} \label{eq:majority}
\(B_K(H)-w-1\)^{n-2}\(\(B_K(H)-w-1\)w+\(U_K(H)-w\)w\)
\end{equation}
such vectors of rank $1$ associated with $(i_0,i_1)$. So, by~\eqref{BK},
 \eqref{eq:twop}, \eqref{eq:majority} and Lemma~\ref{lem:units}, the number of such vectors of rank $1$ associated with an exponent vector $\mathbf{k}$ with $k_{i_0}=t$, $k_{i_1}=\pm t$ for $t$ a non-zero integer is
\begin{equation} \label{eq:majority1}
\begin{split}
\frac{n(n-1)}{2}wC_1(K)^{n-1}H^{d(n-1)}&(\log H)^{r(n-1)}\\
& +O\(H^{d(n-1)}(\log H)^{r(n-1)-1}\).
\end{split}
\end{equation}

It remains to estimate the number of such vectors of multiplicative rank $1$ associated with an exponent vector $\mathbf{k}$ with $k_{i_0}=t_1$ and $k_{i_1}=t_2$ with $t_1\neq \pm t_2$ and $t_1$ and $t_2$ non-zero  integers.  
Let $\nu_1, \nu_2 \in \cB_K(H)$ be  associated with $t_1,-t_2$ respectively. In this case
$$
\nu^{t_1}_1=\nu^{t_2}_2.
$$

We first consider the case when $t_1$ and $t_2$ are of opposite signs. Then, $\nu_1$ and $\nu_2$ are units in the ring of algebraic integers of $K$, and so by Lemma~\ref{lem:units} the number of such vectors is
\begin{equation} \label{eq:extra1}
O\((\log H)^{2r}B_K(H)^{n-2}\).
\end{equation}

It remains to consider the case when $t_1$ and $t_2$ are both positive. Without loss of generality, we assume that $0<t_1 <  t_2$, and also $t_2 \ll \log H$ by Lemma~\ref{lem:exponent}. 

If $t_2 = 2t_1$, then $\nu_1$ is determined by $\nu_2^2$ up to a root of unity contained in $K$, and also we have $\wH(\nu_2) \le H^{1/2}$. 
So, the number of such pairs $(\nu_1,\nu_2)$ is $O(H^{d/2}(\log H)^r)$ by using~\eqref{BK}, and thus the number of such vectors of rank $1$ is 
\begin{equation} \label{eq:t22t1}
O\(H^{d/2}(\log H)^rB_K(H)^{n-2}\). 
\end{equation}

If $t_1$ divides $t_2$ and $t_2/t_1 \ge 3$, then we have $\wH(\nu_2) \le H^{1/3}$, and so as the above the number of such vectors of rank $1$ is 
\begin{equation} \label{eq:t23t1}
O\(H^{d/3}(\log H)^{r+1}B_K(H)^{n-2}\). 
\end{equation}

Now, we assume that $t_1$ does not divide $t_2$. 
 Let $t$ be the greatest common divisor of $t_1$ and $t_2$. Note that $t_1/t \ge 2$ and $t_2/t \ge 3$. 
 Put
\begin{equation} \label{eq:nu1nu2}
\gamma=\nu^{t_1}_1=\nu^{t_2}_2,
\end{equation}
and let $\beta$ be a root of $x^{t_1t_2}-\gamma$. Observe that
$$
\beta^{t_1}=\eta_1\nu_2\quad \text{and}\quad \beta^{t_2}=\eta_2\nu_1
$$
for some $t_1t_2$-th roots of unity $\eta_1$ and $\eta_2$. There exist integers $u$ and $v$ with $ut_1+vt_2=t$, and so
$$
\beta^t=\beta^{t_1u}\beta^{t_2v}
=\eta_1^u\nu_2^u\eta_2^v\nu^v_1=\eta\alpha
$$
for $\eta$ a $t_1t_2$-th root of unity and $\alpha$ an algebraic integer of $K$. Therefore
\begin{equation} \label{eq:theta}
(\eta\alpha)^{t_2/t}=\beta^{t_2}=\eta_2\nu_1,
\end{equation}
and so
\begin{equation} \label{eq:Htheta}
\wH(\alpha)^{t_2/t}=\wH(\nu_1).
\end{equation}
Since $\wH(\nu_1)\leq H$, we see, from~\eqref{eq:theta} and~\eqref{eq:Htheta}, that $\nu_1$ is determined up to a $t_1t_2$-th root of unity, 
by an algebraic integer of $K$ of height at most $H^{t/t_2}\le H^{1/3}$. Thus, by~\eqref{BK} and Lemma~\ref{lem:exponent}, 
the number of such pairs $(\nu_1,\nu_2)$ is $O(H^{d/3}(\log H)^{r+4})$, hence the number of such vectors of rank $1$ is
\begin{equation} \label{eq:extra2}
O\(H^{d/3}(\log H)^{r+4}B_K(H)^{n-2}\).
\end{equation}
Thus, by~\eqref{BK}, \eqref{eq:majority1}, \eqref{eq:extra1}, \eqref{eq:t22t1}, \eqref{eq:t23t1} and~\eqref{eq:extra2}, we get 
\begin{equation} \label{eq:MnK2}
\begin{split}
L_{n,K,1}(H) =\frac{n(n-1)}{2}wC_1&(K)^{n-1}  H^{d(n-1)}(\log H)^{r(n-1)} \\
& +O\(H^{d(n-1)}(\log H)^{r(n-1)-1}\).
\end{split}
\end{equation}

The estimate~\eqref{MnK} now follows from~\eqref{eq:MnKH}, \eqref{eq:MnK1} and~\eqref{eq:MnK2}.

Finally, assume that $K$ is the rational number field $\Q$ or an imaginary quadratic field. Then, $r=0$, and so $B_K(H) = C_1(K)H^d+O(H^{d-1})$ by \eqref{BK0}. Repeating the above process, we obtain
$$
L_{n,K,0}(H) =  nwC_1(K)^{n-1}H^{d(n-1)} + O(H^{d(n-1)-1})
$$
and 
$$
L_{n,K,1}(H) = \frac{n(n-1)}{2}wC_1(K)^{n-1}  H^{d(n-1)} + O\(H^{d(n-3/2)}\),
$$
where the second error term comes from \eqref{eq:t22t1} (and also~\eqref{eq:majority} when $d=2$). 
Hence, noticing~\eqref{eq:MnKH} and $d=1$ or $2$, we obtain~\eqref{MnQ}.

\subsection{Proof of Theorem~\ref{thm:MnK2}} 

By~\eqref{MnK=} and~\eqref{M*nKm}, we have
\begin{equation} \label{eq:M*nK1234} 
\begin{split}
L^*_{n,K}(H)=  L^*_{n,K,0}&(H) + L^*_{n,K,1}(H) \\
& +O\(H^{2d(n-1)-d}\exp(c_2\log H/\log\log H)\).
\end{split}
\end{equation}

As in the proof of Theorem~\ref{thm:MnK}, we obtain, by using~\eqref{B*K} in place of~\eqref{BK},
\begin{equation} \label{eq:M*nK1}
L^*_{n,K,0}(H)=nwC_2(K)^{n-1}H^{2d(n-1)}+O\(H^{2d(n-1)-1}(\log H)^{\sigma(d)}\),  
\end{equation}
where $\sigma(1)=1$ and $\sigma(d)=0$ for $d>1$. 

Similarly, we find that
\begin{equation} \label{eq:M*nK2}
\begin{split}
L^*_{n,K,1}(H)=n(n-1)wC_2(K)^{n-1}&H^{2d(n-1)}\\
& +O\(H^{2d(n-1)-1}(\log H)^{\sigma(d)}\), 
\end{split} 
\end{equation}
where the main difference from the proof of~\eqref{eq:MnK2} is that the contribution from the exponent vectors $(k_{i_0},k_{i_1})$ equal to $(t,t)$ is the same as when $(k_{i_0},k_{i_1})$ is equal to $(t,-t)$.

The desired result now follows from~\eqref{eq:M*nK1234}, \eqref{eq:M*nK1} and~\eqref{eq:M*nK2} by noticing that 
$$
L^*_{2,K}(H)=  L^*_{2,K,0}(H) + L^*_{2,K,1}(H).
$$

\subsection{Proof of Theorem~\ref{thm:Mnd}} 

We  first establish~\eqref{Mnd}. By~\eqref{Mnd=} and~\eqref{Mndm}, we have
\begin{equation} \label{eq:Mnd12}
\begin{split}
M_{n,d}(H)=M_{n,d,0}&(H)+M_{n,d,1}(H) \\
&+O\(H^{d^2(n-1)-d}\exp(c_1\log H/\log\log H)\).
\end{split}
\end{equation}

Note that each such vector $\pmb{\nu}$ of multiplicative rank 0 has a coordinate which is a root of unity of degree $d$. So, in view of the definition of $w_0(d)$ in~\eqref{eq:w0w1} we have 
$$
nw_0(d)\(A_d(H)-w_0(d)\)^{n-1} \le M_{n,d,0}(H)\le nw_0(d)A_d(H)^{n-1}, 
$$
and thus by~\eqref{Ad} and~\eqref{eq:w0w1}, 
\begin{equation} \label{eq:Mnd1}
\begin{split}
M_{n,d,0}(H) = nw_0(d)C_5(d)^{n-1}&H^{d^2(n-1)} \\
& + O\(H^{d^2(n-1)-d}(\log H)^{\rho(d)}\). 
\end{split}
\end{equation} 
We remark that $M_{n,d,0}(H) =0$ if $w_0(d)=0$. 

Moreover, arguing as in the proof of Theorem~\ref{thm:MnK}, 
 we find that the main contribution to $M_{n,d,1}(H)$ comes from vectors associated with an exponent vector $\mathbf{k}$ which has two non-zero components one of which is $t$ and the other of which is $\pm t$ with $t$ a non-zero integer. Notice that
  the number $U_d(H)$ of algebraic integers which are units of degree $d$ and height at most $H$ satisfies (by using Lemma~\ref{lem:H0})
\begin{equation} \label{eq:Ud}
U_d(H)=O\(H^{d(d-1)}\). 
\end{equation} 
 We then deduce from~\eqref{Ad}, \eqref{eq:w0w1},  \eqref{eq:Ud} and Lemma~\ref{lem:special} that
\begin{equation} \label{eq:Mnd2}
M_{n,d,1}(H)=n(n-1)C_5(d)^{n-1}H^{d^2(n-1)}
 +O\(H^{d^2(n-1)-d/2}\);
\end{equation}
if furthermore $d=2$ or $d$ is odd, then 
\begin{equation} \label{eq:Mnd2'}
M_{n,d,1}(H)=n(n-1)C_5(d)^{n-1}H^{d^2(n-1)}
 +O\(H^{d^2(n-1)-d}\log H\). 
\end{equation}
Here, we need to note that for an algebraic integer $\alpha$ of degree $d$ and a root of unity $\eta \ne \pm 1$, $\alpha\eta$ might not be of degree $d$. 

The desired asymptotic formula~\eqref{Mnd} now follows from~\eqref{eq:Mnd12}, \eqref{eq:Mnd1} 
and~\eqref{eq:Mnd2}. In order to show~\eqref{Mnd'}, we use~\eqref{eq:Mnd2'} instead of~\eqref{eq:Mnd2}. 
Besides,~\eqref{M2d} follows from~\eqref{eq:Mnd1} and~\eqref{eq:Mnd2'} by noticing that 
$$
M_{2,d}(H) = M_{2,d,0}(H) + M_{2,d,1}(H). 
$$

Finally, we   prove~\eqref{M*nd}, \eqref{M*nd'} and~\eqref{M*2d}. By~\eqref{Mnd=} and~\eqref{M*ndm}, we have 
\begin{equation} \label{eq:M*nd1234}
\begin{split}
M^*_{n,d}(H)=   &M^*_{n,d,0}(H)  + M^*_{n,d,1}(H) \\
&  +O\(H^{d(d+1)(n-1)-d}\exp(c_2\log H/\log\log H)\).
\end{split}
\end{equation}

As before, we have, by using~\eqref{A*d},
\begin{equation} \label{eq:M*nd1}
\begin{split}
M^*_{n,d,0}(H)=nw_0(d)C_6(d)^{n-1}&H^{d(d+1)(n-1)} \\
&+ O\(H^{d(d+1)(n-1)-d}(\log H)^{\vartheta(d)}\).
\end{split}
\end{equation} 
As in~\eqref{eq:Mnd2} and~\eqref{eq:Mnd2'}, we find that 
\begin{equation} \label{eq:M*nd2}
\begin{split}
M^*_{n,d,1}(H)=2n(n-1)C_6&(d)^{n-1}H^{d(d+1)(n-1)} \\
&+O\(H^{d(d+1)(n-1)-d/2}\log H\); 
\end{split}
\end{equation}
if furthermore $d=2$ or $d$ is odd, we have 
\begin{equation} \label{eq:M*nd2'}
\begin{split}
M^*_{n,d,1}(H)=2n(n-1)&C_6(d)^{n-1}H^{d(d+1)(n-1)} \\
&+O\(H^{d(d+1)(n-1)-d}(\log H)^{\vartheta(d)}\). 
\end{split}
\end{equation}

So,~\eqref{M*nd} follows from~\eqref{eq:M*nd1234},  \eqref{eq:M*nd1} and~\eqref{eq:M*nd2}; 
then using~\eqref{eq:M*nd2'} instead of~\eqref{eq:M*nd2} gives~\eqref{M*nd'}. 
In order to deduce ~\eqref{M*2d}, we apply~\eqref{eq:M*nd1} and~\eqref{eq:M*nd2'} 
and notice that 
$$
M^*_{2,d}(H) = M^*_{2,d,0}(H) + M^*_{2,d,1}(H). 
$$

\section{Lower Bound} %Section 5
\label{sec:low}

In this section, we shall prove that~\eqref{MnKm} is sharp, apart from a factor $H^{o(1)}$, 
when $n=s+1$ is even and $K=\Q$. 

We need the following slight extension of~\cite[Lemma~2.3]{MS}. 

\begin{lem}\label{lem:prod=}
Let $k$ and $q$ be integers with $k\geq 2$ and $q \geq 2$.  Let  $\pmb{\gamma}=(\gamma_1, \ldots, \gamma_k) $ with $\gamma_1, \ldots, \gamma_k$
positive real numbers. 
Then, there exists a positive number 
$\Gamma(q,\pmb{\gamma}) $ such that for $T \to \infty$, we have
$$
\ssum_{\substack{ a_1\cdots a_k= b_1\cdots b_k\\
\gcd(a_ib_i, q) = 1\\
1\le a_i,b_i \le T^{\gamma_i}\\i=1, \ldots, k}} \,\, 1 \quad \sim \quad \Gamma(q,\pmb{\gamma}) T^{\gamma}(\log T)^{(k-1)^2}, 
$$
where $\gamma = \gamma_1 + \cdots +\gamma_k$. 
\end{lem}

\begin{proof} 
The proof proceeds along the same lines as in the proof of~\cite[Lemma~2.3]{MS}. 
The only difference is that the primes $p$ which divide $ q$ are now excluded from 
the Euler products that appear in~\cite{MS}. 
\end{proof} 

We show that apart perhaps from the factor $\exp(c_1 \log H/ \log\log H)$ the estimate~\eqref{MnKm} 
in Proposition~\ref{prop:MnKm} is sharp  when 
$n$ is even, $s = n-1$ 
and $K=\Q$. 
\begin{thm} 
\label{thm:MnQm}
Let  $n=2k$, where $k$ is an integer with $k>1$. Then, for sufficiently large $H$, there exists a positive number $c$ depending on 
$n$ such that
\begin{equation} \label{MnQm}
L_{n,\Q,n-1}(H) \ge  cH^{k}(\log H)^{(k-1)^2}.
\end{equation}
\end{thm} 

\begin{proof}
 Fix $n-2$ distinct odd primes $p_i$, $q_i$, $i=2, \ldots, k$. Given positive integers $a_1,\ldots,a_k,b_1,\ldots,b_k$, we first set 
$$
\nu_1= 2 p_2\cdots p_k a_1 \mand
\nu_{k+1} = 2 q_2\cdots q_k b_1.
$$
After this we set
$$
\nu_i = q_i a_i  \mand
\nu_{k+i} = p_i b_{i} , \quad i = 2, \ldots, k. 
$$
Clearly, if $a_1\cdots a_k = b_1\cdots b_k$
 with $\gcd(a_ib_i, 2 p_2q_2\cdots p_kq_k) =1$ for any $2\le i \le k$, then the integer vector $\pmb{\nu}=(\nu_1,\ldots,\nu_n)$ is multiplicatively dependent of rank $n-1$ by noticing that $\nu_1\cdots \nu_k = \nu_{k+1}\cdots \nu_n$ and that there is no non-empty subset $\{i_1,\ldots ,i_m\}$ of $\{1, \ldots ,n\}$ of size less than $n$ for which 
\begin{equation} \label{eq:mult1}
\nu_{i_1}^{j_{i_1}}\cdots \nu_{i_m}^{j_{i_m}}=1,
\end{equation}
with $j_{i_1}, \ldots ,j_{i_m}$ non-zero integers.

For sufficiently large $H$, 
we choose such integers $a_i, b_i \le c_1H$
 for some positive number $c_1$ depending only on the above fixed primes such that we have $|\nu_i| \le H$ for each $1\le i \le n$. Then, each such vector $\pmb{\nu}$ contributes to $L_{n,\Q,n-1}(H)$. Now  applying Lemma~\ref{lem:prod=} to count such vectors (taking $T=c_1H$ and $\gamma_i=1$ for each $i=1, \ldots, k$), we derive
$$
L_{n,\Q,n-1}(H) \ge  cH^{k}(\log H)^{(k-1)^2}, 
$$
where $c$ is a positive number depending on $n$. 
\end{proof}

\section{Comments} %Section 6

It might be of interest to investigate in more detail how tight our bounds are in Propositions~\ref{prop:MnKm} 
and~\ref{prop:Mndm}. In Section~\ref{sec:low} we have taken an initial step in this direction. 

It would be interesting to study  multiplicatively dependent vectors of polynomials 
over finite fields. In this case the degree plays the role of the height.  
While we expect that most of our results can be translated to this case many tools need to be developed and this should be of independent interest.

\section*{Acknowledgements}
%The authors would like to thank the referee for
%a careful reading of the paper and valuable comments.

The first author was supported in part
by   Gruppo Nazionale per le Strutture Algebriche, Geometriche 
e le loro Applicazioni from Istituto Nazionale di Alta Matematica ``F. Severi''.
The research of the second and third authors was supported by the Australian
Research Council Grant DP130100237.
The research of the fourth author was supported in part by the Canada Research Chairs Program and by Grant A3528 from the Natural Sciences and Engineering Research Council of Canada.

\end{document}